\documentclass[11pt]{article}   
\usepackage{latexsym}        
\usepackage{amssymb}         
\usepackage{amscd}           
\usepackage[french]{babel}
\usepackage{amssymb,amsthm,amsmath, amsfonts}

\input epsf

\def\C{{\Bbb C}}
\def\P{{\Bbb P}}
\def\F{{\Bbb F}}

\newtheorem{defi}{D\'{e}finition}[section]
\newtheorem{pro}[defi]{Proposition}

\newtheorem{lem}[defi]{Lemme}
\newtheorem{thm}[defi]{Th\'{e}or\`{e}me}

\newtheorem{exa}[defi]{Exemple}

\newtheorem{rem}[defi]{Remarque}

\swapnumbers

\begin{document}
\title{Fibr\'es  de Schwarzenberger\\
 et fibr\'es logarithmiques  g\'en\'eralis\'es}
\author{Jean Vall\`es}
\date{}

\maketitle
\begin{abstract} 
On propose  une g\'en\'eralisation   des 
fibr\'es  logarithmiques et des fibr\'es de Schwarzenberger 
 sur $\P^n=\P^n(\C)$ en rang plus grand que $n$. 
Les premiers sont associ\'es \`a des ensembles finis de points de $\P^{n\vee}$ et les seconds \`a des courbes de degr\'e plus grand que $n$ sur $\P^{n\vee}$. Sur le plan projectif, nous montrons que deux fibr\'es logarithmiques g\'en\'eralis\'es sont isomorphes si et seulement s'ils sont associ\'es au m\^eme ensemble de points ou bien si les deux ensembles de points sont sur une m\^eme courbe de degr\'e \'egal au rang des fibr\'es.

\smallskip

\begin{center}
\textbf{Abstract}
\end{center}
\hskip5mm We propose a generalization of logarithmic and Schwarzenberger bundles over $\P^n=\P^n(\C)$ when the rank is greater than $n$.
The first ones are associated to finite sets of points on $\P^{n\vee}$ and the second ones to curves with degree greater than $n$ on $\P^{n\vee}$. On the projective plane we show that two logarithmic bundles are isomorphic if and only if they are associated to the same set of points or if the two sets of points belong to a curve of degree equal to the rank of the considered bundles.

\end{abstract}
{\small Mots cl\'es : fibr\'es logarithmiques g\'en\'eralis\'es, hyperplans instables, vari\'et\'e d'incidence.}
\section{Introduction}
Les fibr\'es de Schwarzenberger (introduits par Schwarzenberger \cite{S}) et les fibr\'es logarithmiques (introduits en toute g\'en\'eralit\'e  par Deligne  \cite{De} et \'etudi\'es sur $\P^n$  par Dolgachev et Kapranov \cite{DK}), poss\'edant une g\'eom\'etrie tr\`es riche,  ont fait et font encore l'objet de nombreux travaux. Ce sont des fibr\'es vectoriels  de rang $n$ sur $\P^n$. Dans cet article nous en proposons  une g\'en\'eralisation  ainsi qu'une \'etude en rang plus grand que $n$.

\smallskip

Nous rappelons pour commencer (section \ref{rappel}), les  d\'efinitions des fibr\'es de Schwarzenberger, des fibr\'es logarithmiques et  le th\'eor\`eme ``de type Torelli'' qui fait le lien entre les deux. Ce th\'eor\`eme a \'et\'e  \'enonc\'e et prouv\'e par Dolgachev et Kapranov (\cite{DK}, thm. 7.2). 

\smallskip

\`A un groupe de point $Z$ de $\P^{n\vee}$ en position lin\'eaire g\'en\'erale, on associe un fibr\'e de rang $n$ sur $\P^n$, appel\'e fibr\'e logarithmique. Nous montrons (section \ref{premiere}) que ce fibr\'e vectoriel est l'image directe sur $\P^{n}$ de l'image inverse du faisceau d'id\'eaux $\mathcal{J}_Z(1)$ sur la vari\'et\'e d'incidence points-hyperplans de $\P^{n}$  (prop. \ref{ext}). Ceci
donne  une description globale des fibr\'es logarithmiques en fonction du groupe de points $Z$. Qui plus est, cette description globale apporte avec elle une g\'en\'eralisation naturelle des fibr\'es logarithmiques en rang plus grand que $n$ (en consid\'erant les faisceaux $\mathcal{J}_Z(r+1)$, pour $r\ge 0$). 

\smallskip

Les fibr\'es de Schwarzenberger de rang $n$ sur $\P^n$ sont associ\'es aux courbes rationnelles normales (donc aux courbes non d\'eg\'en\'er\'ees de degr\'e $n$) de $\P^{n\vee}$. Plus g\'en\'eralement, nous associons des fibr\'es de Steiner de rang  $n+r$ (avec $r\ge 0$)   \`a des courbes de degr\'es $n+r$ sur $\P^{n\vee}$ (section \ref{seconde}).

\smallskip

Nous revenons ensuite (section \ref{instables}) sur la d\'efinition et caract\'erisation cohomologique des hyperplans instables d'un fibr\'e de Steiner (lorsque le rang \'egale $n$, voir \cite{AO} ou \cite{Va}). En particulier, nous \'etudions les hyperplans instables des logarithmiques et Schwarzenberger  g\'en\'eralis\'es. 

\smallskip

Enfin, dans la derni\`ere partie (section \ref{plan}), nous montrons (th\'eor\`eme \ref{teo}) un th\'eor\`eme de ``type Torelli'' qui g\'en\'eralise sur le plan projectif celui de Dolgachev et Kapranov. Cette  g\'en\'eralisation dit en substance :

\smallskip

{\it  Soient $Z$ et $Z^{'}$ deux groupes de points de $\P^{2\vee}$ tels que  $E_{r+1}(Z)\simeq E_{r+}(Z^{'})$ (o\`u $E_{r+1}(Z)$  d\'esigne le fibr\'e logarithmique g\'en\'eralis\'e associ\'e \`a $Z$). Alors un des deux cas  suivants se produit :\\
1) $Z=Z^{'}.$\\
2) $Z$ et $Z^{'}$ sont sur une m\^eme courbe  $X_{r+2}$ de degr\'e $r+2$ et il existe un faisceau $\mathcal{L}$ de rang $1$ sur $X_{r+2}$ tel que $E_{r+1}(Z)\simeq E(X_{r+2}, \mathcal{L})$ (o\`u $E(X_{r+2}, \mathcal{L})$  d\'esigne un fibr\'e de Schwarzenberger g\'en\'eralis\'e).}

\smallskip

{\small Ce travail a \'et\'e r\'ealis\'e lors de mon s\'ejour dans les locaux  de l'universit\'e Ulisse Dini de Florence, que je remercie pour son hospitalit\'e. Ce s\'ejour fut partiellement financ\'e par GNSAGA-INdAM (Italie) et principalement par le CNRS (France) qui m'accueillait en d\'el\'egation.}

\section{Rappels}
\label{rappel}
Avant toute g\'en\'eralisation, nous rappelons dans cette section les d\'efinitions des fibr\'es logarithmiques, de Steiner et de Schwarzenberger. Pour aller plus loin, les r\'ef\'erences principales sont les articles
\cite{DK} et \cite{S}.

\smallskip

\'Etant donn\'es  une vari\'et\'e projective lisse $X$ et un diviseur \`a croisement normaux $D$ sur  $X$, on  d\'efinit le fibr\'e $\Omega_X(\mathrm{log}(D))$ des $1$-formes diff\'erentielles \`a p\^oles au plus logarithmiques le long de $D$. Lorsque $X$ est l'espace projectif complexe $\P^n$ et $D=\cup_{i=1}^{k} H_i$ une union d'hyperplans
en position lin\'eaire g\'en\'erale, ces fibr\'es sont d\'efinis (par Deligne dans l'article \cite{De}) comme  l'extension particuli\`ere qui suit :
$$0\rightarrow \Omega_{\P^n} \longrightarrow \Omega_{\P^n}(\mathrm{log}(D))\stackrel{\mathrm{res}}\longrightarrow 
\oplus_{i=1}^{k}O_{H_i} \rightarrow 0.$$
L'application $\mathrm{res}$ \'etant donn\'ee localement par  l'application r\'esidu de Poincar\'e (pour une d\'efinition explicite voir \cite{DK}, prop. 2.3). 

\smallskip

Dans la suite du texte on pr\'ef\`erera la notation ``duale'' :
\begin{defi} Soient $ H_1, \cdots, H_k$ des hyperplans de $\P^n$ en position lin\'eaire g\'en\'erale et 
$Z=\lbrace H^{\vee}_1, \cdots, H^{\vee}_k\rbrace \subset \P^{n\vee}$
l'ensemble des points duaux.
On appelle  fibr\'e logarithmique associ\'e \`a $Z$, le fibr\'e 
$$E(Z) :=\Omega_{\P^n}(\mathrm{log}(\cup_{i=1}^{k} H_i)).$$
\end{defi}
Lorsque le cardinal  de $Z\subset \P^{n\vee}$ est plus grand que $n+2$, le fibr\'e logarithmique $E(Z)$  
appartient \`a une famille plus large de fibr\'es, appel\'es fibr\'es de Steiner, dont nous rappelons  la d\'efinition.
\begin{defi}
Soient $n,m,r$ trois entiers avec $n>0$, $r\ge 0$ et $m>0$. On note  ${\cal FS}_{n,n+r,m}$ (resp. ${\cal S}_{n,n+r,m}$) l'ensemble des faisceaux coh\'erents $E$ (resp. des fibr\'es vectoriels)  admettant une r\'esolution du type :
$$\begin{CD}
0 @>>> O_{\P^n}^m(-1) @>>> O_{\P^n}^{m+n+r} @>>> 
E @>>> 0.
\end{CD}
$$
On les appellera \textit{faisceaux  de  Steiner} (resp. \textit{fibr\'es de  Steiner}).
\end{defi}
Par exemple,  lorsque $Z\subset \P^{n\vee}$ est un groupe de points en position lin\'eaire g\'en\'erale de cardinal $k\ge n+2$, Dolgachev et Kapranov ont montr\'e (\cite{DK}, thm. 3.5) que $E(Z)\in {\cal S}_{n,n, k-n-1}$ (pour $k=n+2$, c'est le fibr\'e tangent $T_{\P^n}(-1)$).
 
\smallskip

Les fibr\'es de Schwarzenberger, introduits par Schwarzenberger (\cite{S}, thm. 4), sont eux aussi des fibr\'es de Steiner de rang $n$ sur $\P^n$. De plus, ils sont compl\`etement d\'efinis par la donn\'ee d'un faisceau inversible sur la courbe rationnelle normale 
$C_n\subset \P^{n\vee}$. Ils h\'eritent ainsi d'une g\'eom\'etrie tr\`es riche, provenant des nombreuses propri\'et\'es des courbes rationnelles (voir par exemple \cite{AO}, \cite{DK} et \cite{S}). Pour les d\'efinir on consid\`ere  un espace vectoriel $U$ de dimension deux sur $\C$. On note  $S_i=\mathrm{Sym}^iU$ les puissances sym\'etriques de $U$, 
$(X_0,\cdots, X_n)$ les coordonn\'ees de $\P^n=\P(S_n)$ et $C_n\subset \P^{n\vee}=\P(S_n^{\vee})$ l'image de 
$\P(S_1^{\vee})$ par le morphisme de Veronese. 
\begin{defi}[\cite{S}, prop. 2]
Pour tout entier $m$ tel que $m\ge n $,  le fibr\'e  de Steiner  d\'efini par la suite exacte suivante :
$$ \begin{CD}
0 @>>> S_{m-n}\otimes O_{\P(S_n)}(-1) @>M>>S_{m}\otimes O_{\P(S_n)}
@>>> E_{m}(C_n)
@>>>0
\end{CD}
$$
o\`u 
$$  {}^tM=\left (
             \begin{array}{cccccccc}
              X_0 & X_1 & \cdots  & X_n & 0 & \cdots  & \cdots &0\\
              0 & X_0 & X_1 &\cdots      &X_n & 0 & & \vdots \\
               \vdots&  0  &  X_0 & X_1 &\cdots      &X_n & \ddots & \vdots   \\
               \vdots &   & \ddots &\ddots & \ddots&  &\ddots&0\\
                0&\cdots  & \cdots &  0 &X_0 & X_1 & \cdots  & X_n
             \end{array}
      \right ).
$$
est appel\'e 
\textit{fibr\'e de Schwarzenberger associ\'e \`a la courbe} $C_n$.
\end{defi}
Lorsque un groupe de points $Z$ de longueur $k\ge n+2$ appartient \`a une courbe rationnelle normale $C_n\subset  \P^{n\vee}$,
le fibr\'e logarithmique $E(Z)$ est  un fibr\'e de Schwarzenberger. Alors, la correspondance 
$Z \rightsquigarrow E(Z)$ n'est pas injective. R\'eciproquement, si la correspondance n'est pas injective alors le groupe de points $Z$ appartient \`a une courbe rationnelle normale. Ce r\'esultat d'injectivit\'e, que nous \'enon\c{c}ons plus formellement ci-dessous, a \'et\'e  prouv\'e par Dolgachev et Kapranov (\cite{DK}, thm. 7.2 pour $k\ge 2n+3$, et 
\cite{Va}, cor. 3.1 pour $k\ge n+2$, en \'etudiant les hyperplans instables de $E(Z)$).
\begin{thm}
Soient $Z$ et $Z^{'}$ deux groupes de points de longueur $k\ge n+2$ en position lin\'eaire g\'en\'erale dans $\P^{n\vee}$ pour lesquels, on a 
$E(Z)\simeq E(Z^{'})$. Alors, un des deux cas suivants se produit :\\
1) $Z=Z^{'}$.\\
2) $Z$ et $Z^{'}$ sont sur une m\^eme courbe rationnelle normale $C_{n}$ et $E(Z)\simeq E_{k-2}(C_n)$.
\end{thm}

\section{Construction \'el\'ementaire I}
\label{premiere}
Consid\'erons la vari\'et\'e d'incidence points-hyperplans de $\P^{n}$ :
$$
\begin{CD}
 \F @>q>>    \P^{n \vee}\\
 @VpVV \\
\P^{n}
\end{CD}
$$
Les points de $\P^n$ et $\P^{n\vee}$ seront not\'es $x$ et $H^{\vee}$ (resp. $x$ et $l^{\vee}$ pour $n=2$), $H$ et $x^{\vee}$ (resp. $l$ et $x^{\vee}$ pour $n=2$) d\'esignerons les hyperplans de $\P^n$ et $\P^{n\vee}$.
 On dira qu'un groupe de points $Z$
est en position $(r+1)$-g\'en\'erale
dans $\P^{n\vee}$ s'il n'est pas contenu dans une hypersurface de degr\'e $r+1$ (en particulier il est de longueur $\ge \binom{n+r+1}{n}$) et si chaque sous-groupe de longueur $s\le \binom{n+r}{n-1}$ contenu dans un hyperplan $x^{\vee}$, impose $s$ conditions lin\'eairement ind\'ependantes sur les hypersurfaces de degr\'e $(r+1)$ de $x^{\vee}$.
En d'autres termes, la dimension $h^0(\mathcal{J}_Z(r+1)\otimes O_{x^{\vee}})$ est constante pour tout $x^{\vee}$. Cette hypoth\`ese ad hoc correspond \`a la condition de position lin\'eaire g\'en\'erale (requise pour les fibr\'es logarithmiques) lorsque $r=0$.
\begin{pro}
\label{pro31}
Soient  $Z \subset \P^{n\vee}$ un groupe de points en position $(r+1)$-g\'en\'erale
et $\mathcal{J}_Z$ son faisceau d'id\'eaux. Alors, 
$ [p_*q^*\mathcal{J}_Z(r+1)]^{\vee}(-1) $ est un fibr\'e de Steiner de rang $\binom{n+r}{n-1}$ sur $\P^n$.
\end{pro}
\begin{proof}
Soit  $x^{\vee}$ un hyperplan de $\P^{n\vee}$. On note $ \vert Z\cap x^{\vee} \vert$ le nombre de points  de l'intersection et  $\mathcal{J}_{Z\cap x^{\vee}}$ le faisceau d'id\'eaux de $Z\cap x^{\vee}$ dans 
$x^{\vee}$. On a alors :
\begin{equation}
\label{eq}
 \mathcal{J}_{Z}(r+1)\otimes O_{x^{\vee}}=O_{Z\cap x^{\vee}} \oplus \mathcal{J}_{Z\cap x^{\vee}}(r+1).
\end{equation}
Comme  $Z$ est en position $(r+1)$-g\'en\'erale, la dimension $h^0(\mathcal{J}_Z(r+1)\otimes O_{x^{\vee}})$ est constante pour tout $x^{\vee}$. En particulier, lorsque $x^{\vee}$ ne coupe pas $Z$, l'\'egalit\'e ci-dessus donne, 
$$h^0(\mathcal{J}_{Z}(r+1)\otimes O_{x^{\vee}})=h^0(O_{x^{\vee}}(r+1))=\binom{n+r}{n-1}. $$ 
Autrement dit, 
$p_{*}q^{*}{\mathcal J}_{Z}(r+1)$ est un fibr\'e sur $\P^n$ de rang $\binom{n+r}{n-1}$ et $R^1p_{*}q^{*}{\mathcal J}_{Z}(r+1)=0$.

\smallskip

Afin de montrer que $p_{*}q^{*}{\mathcal J}_{Z}(r+1)$ est un fibr\'e de Steiner, nous consid\'erons maintenant  la r\'esolution minimale de $\F$ dans le produit $\P^n\times \P^{n\vee}$ 
$$\begin{CD}
0 @>>> O_{\P^n\times \P^{n\vee}}(-1,-1)@>>> O_{\P^n\times \P^{n\vee}} @>>>  O_{\F} @>>> 0.
\end{CD}
$$
On tensorise cette suite  par $pr_2^*\mathcal{J}_Z(r+1)$ (o\`u $pr_2$ est la projection de $\P^n\times \P^{n\vee}$ sur le second facteur) et on prend son image directe sur $\P^n$. Comme $h^0(\mathcal{J}_Z(r+1))=0$
et $R^1p_{*}q^{*}{\mathcal J}_{Z}(r+1)=0$, on a bien
{\small
$$\begin{CD}
0 @>>> p_{*}q^{*}{\mathcal J}_{Z}(r+1)@>>> H^{1}({\mathcal J}_{Z}(r))\otimes O_{\P^{n}}(-1) @>>> H^{1}({\mathcal J}_{Z}(r+1))\otimes O_{\P^{n}} @>>> 0.
\end{CD}
$$}
\end{proof}
\begin{pro}
\label{ext}
Soient  $Z=\lbrace H_1^{\vee},\cdots, H_k^{\vee} \rbrace \subset \P^{n\vee}$ un groupe de points en position $(r+1)$-g\'en\'erale
et $E$ un fibr\'e vectoriel apparaissant dans une extension  
$$\begin{CD}
0  @>>> (S^{r+1}\Omega_{\P^n})(r) @>>> E @>>> \oplus_{i=1}^{k}O_{H_i}
@>>> 0.
\end{CD}
$$
Alors, $E\simeq [p_*q^*\mathcal{J}_Z(r+1)]^{\vee}(-1).$
\end{pro}
\begin{rem} \upshape
Lorsque $r=0$, ce sont les fibr\'es logarithmiques de rang $n$ sur $\P^n$.
\end{rem}
\begin{proof}
Comme $\mathcal{E}xt^1(O_{H_i}, O_{\P^n})=O_{H_i}(1)$, on obtient en dualisant l'extension  d\'efinissant $E$ la suite exacte suivante :
$$\begin{CD}
0 @>>> E^{\vee}(-1)  @>>> S^{r+1}\Omega_{\P^n}^{\vee}(-1) @>>> \oplus_{i=1}^{k}O_{H_i}
@>>>  0.
\end{CD}
$$
Pour tout $1\le i\le k$ la fl\`eche surjective 
$\begin{CD}
 S^{r+1}\Omega_{\P^n}^{\vee}(-1) @>>> O_{H_i}
\end{CD}
$
se factorise par la restriction 
$$\begin{CD}
S^{r+1}\Omega_{\P^n}^{\vee}(-1) @>>> S^{r+1}(\Omega_{\P^n}^{\vee}(-1)\otimes O_{H_i}) @>>> O_{H_i}.
\end{CD}
$$
Or ce fibr\'e restreint est 
$$S^{r+1}(\Omega_{\P^n}^{\vee}(-1)\otimes O_{H_i})=O_{H_i}\bigoplus_{1\le j\le r+1}S^{j}(\Omega_{H_i}^{\vee}(-1)).$$
Comme $\mathrm{Hom}(S^{j}(\Omega_{H_i}^{\vee}(-1)), O_{H_i})=0$, on en d\'eduit que 
$$ \mathrm{Hom}(S^{r+1}\Omega_{\P^n}^{\vee}(-1),\oplus_{i=1}^{k}O_{H_i})=
\mathrm{Hom}(\oplus_{i=1}^{k}O_{H_i},\oplus_{i=1}^{k}O_{H_i}). $$ 
Par cons\'equent, deux extensions g\'en\'erales (i.e. deux fibr\'es vectoriels) correspondent \`a deux matrices diagonales $M_{\phi}$, $M_{\psi}$ de rang maximal.
 On en d\'eduit que  le diagramme  suivant commute :
$$
\begin{CD} 
0  @>>>  {\rm ker}(\phi) @>>>    S^{r+1}\Omega_{\P^n}^{\vee}(-1)@>\phi>> 
\oplus_{i=1}^{k}O_{H_i} @>>>    0\\
@. @V{\simeq}VV  @V{\simeq}VV
 @VM_{\psi}M_{\phi}^{-1}VV\\
0  @>>> {\rm ker}(\psi) @>>>   S^{r+1}\Omega_{\P^n}^{\vee}(-1)  @>\psi>>  \oplus_{i=1}^{k}O_{H_i}
 @>>> 0. 
\end{CD}
$$
Ce qui prouve l'unicit\'e du noyau lorsqu'il est un fibr\'e.

\smallskip

On montre  que ce noyau unique est $p_*q^*\mathcal{J}_Z(r+1)$. En effet, l'image par  $p_*q^*$ de la suite exacte ci-dessous
$$\begin{CD}
0 @>>> \mathcal{J}_Z(r+1)  @>>> O_{\P^{n\vee}}(r+1) @>>> O_{Z}@>>> 0
\end{CD}
$$
 donne 
$$\begin{CD}
0 @>>> p_*q^*\mathcal{J}_Z(r+1)  @>>> S^{r+1}\Omega_{\P^n}^{\vee}(-1) @>>> \oplus_{i=1}^{k}O_{H_i}
@>>>  0.
\end{CD}
$$
\end{proof}
En particulier, lorsque $r=0$, on retrouve  le fibr\'e logarithmique $E(Z)$ associ\'e \`a $Z$. Plus pr\'ecis\'ement, 
sous les hypoth\`eses de la proposition \ref{ext}, on a montr\'e :
$$ E(Z)^{\vee}(-1)=p_*q^*\mathcal{J}_Z(1).$$
Dans ce sens, les fibr\'es $p_*q^*\mathcal{J}_Z(r+1)$ pour $r\ge 0$ sont bien une g\'en\'eralisation des fibr\'es logarithmiques et la question sur la biunivocit\'e de la correspondance
 $$ Z \rightsquigarrow p_*q^*\mathcal{J}_Z(r+1) $$
reste  l\'egitime. On notera \`a partir de maintenant $E_{r+1}(Z):=[p_*q^*\mathcal{J}_Z(r+1)]^{\vee}(-1)$.

\section{Construction \'el\'ementaire II}
\label{seconde}
Soient $X\subset \P^{n\vee}$ une courbe int\`egre  non d\'eg\'en\'er\'ee de degr\'e $n+r$. Appelons 
$\overline{q}$ le morphisme de projection de $q^{-1}(X)$ sur $X$ (restriction du morphisme $q : \F\rightarrow \P^{n\vee}$) et $\overline{p} $ celui de $q^{-1}(X)$ sur $\P^n$ (restriction du morphisme $p : \F\rightarrow \P^{n}$). Le morphisme $\overline{p}$ est un rev\^etement de degr\'e $(n+r)$ de  $\P^{n}$.
\begin{pro} 
\label{seconstruction}
Soit  $\mathcal{L}$ un faisceau de rang $1$ sur la courbe  $X$ tel que $h^1(\mathcal{L}(-1))=0$. Alors
$E(X,\mathcal{L}):=\overline{p}_{*}\overline{q}^{*}\mathcal{L}$ est un faisceau  de Steiner de rang \'egal au degr\'e de $X$. 
\end{pro}
\begin{rem}\upshape
Ces faisceaux sont des fibr\'es de Steiner lorsque $h^0(\mathcal{L}\otimes O_{x^{\vee}})$ est constant. Ceci impose que le faisceau $\mathcal{L}$ soit inversible mais la courbe peut-\^etre singuli\`ere (exemple \ref{courbesinguilere}). 
Ces fibr\'es g\'en\'eralisent  les fibr\'es de Schwarzenberger qui proviennent des images directes des faisceaux inversibles sur les courbes rationnelles normales.
Arrondo  propose une autre g\'en\'eralisation des fibr\'es de Schwarzenberger (\cite{A}, d\'ef. page 6) qui, aux d\'etails pr\`es, englobe la n\^otre.
\end{rem}
\begin{proof} La dimension relative \'etant nulle on a 
$R^1\overline{p}_{*}\overline{q}^{*}\mathcal{L}=0$.
Par cons\'equent, $\overline{p}_{*}\overline{q}^{*}\mathcal{L}$ est un faisceau de Steiner  sur $\P^n$  avec pr\'esentation :
$$\begin{CD}
0  @>>> H^0(\mathcal{L}(-1))\otimes O_{\P^n}(-1) @>>> 
H^0(\mathcal{L})\otimes O_{\P^n} @>>> \overline{p}_{*}\overline{q}^{*}\mathcal{L} @>>> 0.
\end{CD}
$$
Son  rang est \'egal \`a 
$h^0(\mathcal{L})-h^0(\mathcal{L}(-1))=\mathrm{deg}(X)=n+r.$
\end{proof}
\section{Hyperplans instables}
\label{instables}
Soit  $E\in \mathcal{S}_{n,n+r,m}$ un fibr\'e de Steiner. L'ensemble $W(E)$ de ses  hyperplans instables  est (d\'efinition que Dolgachev \'etend aux faisceaux dans  \cite{Do}) :
 $$ W(E):=\lbrace H^{\vee}\in \P^{n\vee}, H^0(E_H^{\vee})\neq 0\rbrace. $$

\smallskip

Pour le d\'eterminer explicitement, on utilisera l'\'egalit\'e  
$ W(E)=\mathrm{supp}(R^{n-1}q_*p^*E(-n))$
qui provient de la  dualit\'e de Serre  $h^0(E_H^{\vee})=h^{n-1}(E_H(-n))$. 

\smallskip

Lorsque $E=O_{\P^n}\oplus F$, ou bien $E=\Omega_{\P^n}^{\vee}(-1)\oplus F$, ou bien encore lorsque 
$m+n+r>mn$, il est clair que $H^0(E_H^{\vee})\neq 0$ pour tout $H^{\vee}\in \P^{n\vee}$. On s'int\'eressera plut\^ot  aux fibr\'es de Steiner  pour lesquels l'hyperplan g\'en\'eral n'est pas instable. 
\begin{pro}
\label{premiercas}
Soit $E\in {\cal S}_{n,n+r,m}$. La codimension attendue du lieu $W(E)$ des hyperplans instables dans $ \P^{n\vee}$ est  $mn-(m+n+r)+1$.
\end{pro}
\begin{proof}
Cela r\'esulte de la formule de Thom-Porteous appliqu\'ee \`a la r\'esolution du faisceau $R^{n-1}q_*p^*E(-n)$ :
{\small $$\begin{CD}
 H^{n-1}(E(-n-1))\otimes O_{\P^{n\vee}}(-1) @>>> 
H^{n-1}(E(-n))\otimes O_{\P^{n\vee}} @>>> R^{n-1}q_*p^*E(-n) @>>> 0.
\end{CD}
$$ }
\end{proof}

\begin{exa}
{\rm 
Sur $\P^2$, consid\'erons un fibr\'e de Steiner $E$ poss\'edant une r\'esolution du type :
$$\begin{CD}
0 @>>> O_{\P^2}^{r+2}(-1) @>>> O_{\P^2}^{2r+4} @>>> 
E @>>> 0.
\end{CD}
$$
Quitte \`a choisir $E\in {\cal S}_{2,r+2,r+2}$ assez g\'en\'eral, nous pouvons supposer que 
 $h^0(E_l(-2))=0$ pour $l$ g\'en\'erale. Nous avons alors une suite exacte courte :
$$\begin{CD}
0 @>>> O_{\P^{2\vee}}(-1)^{r+2} @>>> O_{\P^{2\vee}}^{r+2} @>>> 
R^1q_*p^*E(-2) @>>> 0.
\end{CD}
$$
Ceci prouve que $W(E)$ est une courbe de degr\'e $r+2$.}
\end{exa}

\begin{pro}
Soient $Z$ un groupe de points de $\P^{n\vee}$ en position $(r+1)$-g\'en\'erale et $E_{r+1}(Z)$ le logarithmique g\'en\'eralis\'e qui lui est associ\'e. Alors 
$Z\subset W(E_{r+1}(Z)).$ 
\end{pro}
\begin{rem} \upshape
Nous verrons (th\'eor\`eme \ref{teo}) que sur le plan projectif     $W(E_{r+1}(Z))=Z$ si et seulement si $h^0({\mathcal J}_{Z}(r+2))=0$.
\end{rem}

\begin{proof}
Soit $H^{\vee}\in Z$. La suite exacte sur $\P^{n\vee}$ reliant les faisceaux d'id\'eaux 
$$\begin{CD}
0 @>>> {\mathcal J}_{Z}(r+1) @>>> {\mathcal J}_{Z\setminus \lbrace H^{\vee} \rbrace}(r+1) @>>> O_{H^{\vee}} @>>> 0
\end{CD}
$$
induit par image r\'eciproque et image directe la suite exacte suivante sur $\P^{n}$
$$\begin{CD}
0 @>>> p_{*}q^{*}{\mathcal J}_{Z}(r+1) @>>> p_{*}q^{*}{\mathcal J}_{Z\setminus \lbrace H^{\vee} \rbrace}(r+1) @>>> O_{H} @>>> 0.
\end{CD}
$$
Il suffit alors de  dualiser  cette suite et de la  tensoriser par $O_{\P^{n}}(-1)$ pour v\'erifier que 
$H^{\vee}\in W(p_{*}q^{*}{\mathcal J}_{Z}(r+1))$.
\end{proof}
On a un  r\'esultat de m\^eme nature pour les fibr\'es de Steiner provenant  de faisceaux inversibles sur une courbe. 
\begin{pro}
Soient $X\subset \P^{n\vee}$ une courbe int\`egre  non d\'eg\'en\'er\'ee de degr\'e $n+r$ et $\mathcal{L}$ un faisceau inversible non sp\'ecial sur $X$. 
Alors,  $X\subset W(E(X, \mathcal{L}(1)).$ 
\end{pro}
\begin{proof}
Il est clair, en appliquant le foncteur   $\overline{p}_*\overline{q}^*$ \`a la suite  
$$\begin{CD}
0 @>>> \mathcal{L}(1-H^{\vee}) @>>> \mathcal{L}(1) @>>> O_{H^{\vee}}
 @>>> 0
\end{CD}
$$
que tout point $H^{\vee}$ de la courbe fournit un hyperplan instable  $H$  de $\overline{p}_*\overline{q}^*\mathcal{L}(1)$.
\end{proof}
Dans $\P^{n\vee}$ avec $n\ge 3$, il semble difficile comme l'a aussi remarqu\'e E. Arrondo (\cite{A}, rem. 2.9) de prouver que les deux ensembles co\"{\i}ncident. En effet, on peut avoir une surface d'hyperplans instables pour un fibr\'e  provenant d'un faisceau inversible sur une courbe de cette surface, comme dans l'exemple ci-dessous :
\begin{exa}
{\rm Soient $C$ une courbe  lisse de degr\'e $4$ de $\P^{3\vee}$  et $\mathcal{L}$ un faisceau inversible sur $C$ de degr\'e \'egal \`a $6$. Le fibr\'e de Steiner associ\'e $E(C,\mathcal{L})$ est un fibr\'e de rang $4$ sur $\P^3$ :
$$\begin{CD}
0 @>>> O_{\P^3}^2(-1) @>>> O_{\P^3}^{6} @>>> 
E(C,\mathcal{L}) @>>> 0.
\end{CD}
$$
Si le plan g\'en\'eral n'est pas instable, $W(E(C,\mathcal{L}))$ est une surface quadrique contenant strictement la courbe $C$.}
\end{exa}
Par contre dans le plan projectif, d\`es lors que la droite g\'en\'erale n'est pas instable, il y a \'egalit\'e.
\begin{pro}
\label{droiteinstable}
Soit $E\in \mathcal{S}_{2,r+2,m}$. Supposons que $W(E)\neq \P^{2\vee}$ et que $W(E)$ contienne une courbe $X_{r+2}$ de degr\'e $r+2$. Alors, 
$W(E)=X_{r+2}$.
\end{pro}
Cette proposition repose sur le lemme qui suit.
\begin{lem}
\label{secdroite}
Soit $E\in \mathcal{S}_{2,r+2,m}$. On suppose qu'il existe une droite $ x^{\vee}$ de $\P^{2\vee}$ qui est  
$(r+3)$-s\'ecante \`a   $W(E)$.
 Alors,  $x^{\vee} \subset W(E)$.
\end{lem}
\begin{proof}
Soient   $l_1,\cdots,l_{r+3}$ des droites instables pour $E$, 
concourantes en un point $x\in \P^2$. Consid\'erons l'\'eclatement $\tilde{\P} $ de $\P^2$ le long de $x$ et le diagrammme d'incidence induit :
$$
\begin{CD}
 \tilde{\P} @>\tilde{q}>>    x^{\vee}\\
 @V\tilde{p}VV \\
x\in \P^2
\end{CD}
$$
Appliquons le foncteur $ \tilde{q}_{*} \tilde{p}^{*}$ \`a la suite exacte 
$$\begin{CD}
0@>>> E^{\vee}
@>>>    O^{m+r+2}_{{\P}^2}@>>>O^m_{{\P}^2}(1)
@>>> 0.
\end{CD} 
$$
 Comme $ \tilde{q}_{*} \tilde{p}^{*}O_{{\P}^2}(1)=
\Omega_{{\P}^{2\vee}}^{\vee}(-1)\otimes O_{x^{\vee}}=O_{x^{\vee}}\oplus O_{x^{\vee}}(1)$,
 on obtient un homomorphisme
$$ \begin{CD}
O^{m+r+2}_{x^{\vee}} @>{M}>>
O^m_{x^{\vee}}\oplus O^m_{x^{\vee}}(1).
   \end{CD}
$$
Par hypoth\`ese, les mineurs maximaux de $M$ 
s'annulent tous aux points $l^{\vee}_1,\cdots,l^{\vee}_{r+3} $ de  $x^{\vee}$. Cependant, un mineur maximal, s'il est non nul, est de degr\'e au plus $(r+2)$ ; il s'annule donc  en au plus $r+2$ points de $x^{\vee}$. On en d\'eduit que 
$M$ n'est pas de rang maximal i.e. que son noyau, qui  par fonctorialit\'e ne peut \^etre que 
$\tilde{q}_{*} \tilde{p}^{*}E^{\vee}$, est de rang sup\'erieur ou \'egal \`a $1$. En d'autres termes 
$h^0(E^{\vee}\otimes O_{l})\neq 0 $ pour tout $l^{\vee}\in x^{\vee}$, c'est-\`a-dire $x^{\vee} \subset W(E)$. 

\end{proof}
\begin{proof}[D\'emonstration de la proposition \ref{droiteinstable}]
Supposons qu'il existe un point $l^{\vee}\in W(E)$ et $l^{\vee}\notin X_{r+2}$. Dans $\P^{2\vee}$, toute droite $x^{\vee}$ passant par 
$l^{\vee}$ est  $(r+3)$-s\'ecante \`a $W(E)$. 
D'apr\`es le lemme \ref{secdroite}, $x^{\vee}\subset W(E)$. Ceci contredit l'hypoth\`ese $W(E)\neq \P^{2\vee}$.
\end{proof}

Nous d\'eveloppons, pour terminer cette partie, le cas  des fibr\'es de Steiner de rang $3$ sur $\P^2$ qui poss\`edent une courbe (cubique lisse ou singuli\`ere) de droites instables. 
Nous consid\'erons dans un premier temps (Famille I,  ci-dessous) les fibr\'es $\overline{p}_*\overline{q}^*O_C(n)$ (o\`u $n\ge 2$, $C$ est une cubique lisse et en reprenant les notations de la proposition \ref{seconstruction}). Nous verrons
(prop. \ref{droiteinst}) que l'hypoth\`ese 
$W(E)\neq \P^{2\vee}$ dans la proposition 
\ref{droiteinstable} n'\'etait pas superflue.

\smallskip

Dans un second temps (Famille II, ci-dessous) nous consid\'erons les fibr\'es de Steiner obtenus comme restrictions planes de fibr\'es de Schwarzenberger sur $\P^3$. Ces derniers sont associ\'es \`a une cubique rationnelle normale $C_3\subset \P^{3\vee}$ et les points de $C_3$ sont leurs plans instables (\cite{Va}, prop. 2.2). 
La restriction d'un de  ces fibr\'es \`a un plan g\'en\'eral $H\subset \P^3$ est un fibr\'e de Steiner de rang $3$ sur $H$. Nous montrons (prop. \ref{schproj}) que les droites instables  du fibr\'e restreint sont les duales des points de la cubique singuli\`ere qui est l'image de $C_3$ par la projection de centre $H^{\vee}$.

\medskip

\textbf{Famille I.} Soit une cubique lisse, not\'ee $C\subset \P^{2\vee}$. Consid\'erons le diagramme d'incidence ``point-droite'' au-dessus de $C$ :
$$ \begin{CD}
C @<\overline{q}<< q^{-1}(C) \subset \F  @>q>> \P^{2\vee} \\
 @. @V\overline{p}VpV \\
@. \P^2
\end{CD}
$$
\begin{pro}
\label{droiteinst}   $\overline{p}_*\overline{q}^*O_C(2)=S^2(\Omega_{\P^2}^{\vee}(-1))$,  en particulier toute droite de $\P^2$ est instable pour  $\overline{p}_*\overline{q}^*O_C(2)$. Par contre, pour $n\ge 3$, 
$W(\overline{p}_*\overline{q}^*O_C(n))=C.$
\end{pro}
\begin{proof} 
En appliquant le foncteur ${p}_*{q}^*$ \`a la suite exacte  
$$
\begin{CD} 
0@>>> O_{\P^{2\vee}}(-1) @>>> O_{\P^{2\vee}}(2) @>>> 
O_C(2) @>>> 0
\end{CD}$$
et en utilisant le fait que $p_*q^{*}O_{\P^{2\vee}}(-1)=R^1p_*q^{*}O_{\P^{2\vee}}(-1)=0$, on obtient 
l'\'egalit\'e  $\overline{p}_*\overline{q}^*O_C(2)=S^2(\Omega_{\P^2}^{\vee}(-1))$.
En particulier 
toute  droite $l$  du plan est  instable car 
$$S^2(\Omega_{\P^2}^{\vee}(-1))\otimes O_l=O_l\oplus O_l(1)\oplus O_l(2).$$
Lorsque $n\ge 3$, il suffit d'apr\`es la proposition \ref{droiteinstable},  de montrer  qu'une droite g\'en\'erale n'est pas instable
pour le fibr\'e $\overline{p}_*\overline{q}^*O_C(n)$. 
L'application naturelle 
$$H^0(\Omega_{\P^2}^{\vee}(-1)\otimes O_l)\rightarrow H^0(N_{l/\P^2}(-1))=\C$$ n'est rien d'autre que l'\'evaluation des formes lin\'eaires sur $\P^{2\vee}$ au point $l^{\vee}\in \P^{2\vee}$. En prenant la puissance sym\'etrique troisi\`eme, la composition 
$$
\begin{CD} 
 O_{l} @>>> O_l\oplus O_l(1)\oplus O_l(2)\oplus O_l(3) @>>> 
O_l
\end{CD}$$
consiste \`a \'evaluer la cubique $C$ au point $l^{\vee}$. Il appara\^{\i}t, en tensorisant par $O_l$  la suite exacte 
$$
\begin{CD} 
0@>>> O_{\P^{2}} @>>> S^3(\Omega_{\P^2}^{\vee}(-1)) @>>> 
\overline{p}_*\overline{q}^*O_C(3) @>>> 0,
\end{CD}$$
que la droite $l$ est instable si et seulement si la composition 
ci-dessus est nulle. On en d\'eduit que $W(\overline{p}_*\overline{q}^*O_C(3))=C.$

\smallskip

Une injection  $O_C(n)\hookrightarrow O_C(n+1)$ induit sur $\P^2$ une injection 
$\overline{p}_*\overline{q}^*O_C(n) \hookrightarrow \overline{p}_*\overline{q}^*O_C(n+1).$ Ceci prouve l'inclusion  
$W(\overline{p}_*\overline{q}^*O_C(n+1)) \subset W(\overline{p}_*\overline{q}^*O_C(n))$. Ainsi, 
pour tout $n\ge 3$, nous avons montr\'e $ W(\overline{p}_*\overline{q}^*O_C(n))=C$.
\end{proof}

\textbf{Famille II.}
 Soient $n\ge 1$ un entier et  $E_{n+3}(C_3)$ 
le fibr\'e de Schwarzenberger sur $\P(S_3)$ associ\'e \`a un diviseur de degr\'e $n+3$ sur  la cubique gauche 
$C_3\subset \P(S_3^{\vee})$. Il est d\'efini par : 
$$
\begin{CD} 
0@>>> S_n\otimes O_{\P(S_3)}(-1) @>M>> S_{n+3}\otimes O_{\P(S_3)} @>>> 
E_{n+3}(C_3) @>>> 0
\end{CD}.$$
Consid\'erons la restriction de  cette suite exacte \`a un  plan g\'en\'eral $H=\P(S_3/\C)\subset \P(S_3)$. 
Notons $\overline{C_3}$ la cubique singuli\`ere (cuspidale ou \`a point double) image de $C_3$ par  la projection de centre  $H^{\vee}$ et $E_{n+3}(\overline{C_3})$ la restriction du fibr\'e de Schwarzenberger $E_{n+3}(C_3)$ :
$$
\begin{CD} 
0@>>> S_n\otimes O_{\P(S_3/\C)}(-1) @>>> S_{n+3}\otimes O_{\P(S_3/\C)} @>>> 
E_{n+3}(\overline{C_3}) @>>> 0.
\end{CD}$$
\begin{pro}
\label{schproj}
\label{courbesinguilere} Soit  $n\ge 2$ un entier, alors
$W(E_{n+3}(\overline{C_3}))=\overline{C_3}.$
\end{pro}
\begin{proof}  
Compte tenu de la d\'ecomposition \'equilibr\'ee de $E_{n+3}(C_3)$ sur une droite g\'en\'erale de $\P^3$ (\cite{ST}, prop. 2.18), une droite g\'en\'erale  de $H$ n'est pas instable pour $E_{n+3}(\overline{C_3})$. 
On sait (\cite{Va}, prop. 2.2) que les plans instables de $E_{n+3}(C_3)$ sont les plans $H_{u,v}$ d'\'equations 
$$ u^3X_0+u^2vX_1+uv^2X_2+v^3X_3=0, \, \, \textrm{pour}\,\, (u,v)\in \P(S_1). $$
L'homomorphisme surjectif $E_{n+3}(C_3) \rightarrow O_{H_{u,v}}$ induit, par restriction \`a $H$,  une surjection 
$$ E_{n+3}(\overline{C_3}) \rightarrow O_{H_{u,v}\cap H}.$$
Par cons\'equent les droites $H_{u,v}\cap H$ du plan $H$, qui sont les droites  duales des points de $\overline{C_3}$,
sont les droites instables de $E_{n+3}(\overline{C_3})$.

\end{proof}

\section{Th\'eor\`eme de ``type Torelli'' sur le plan projectif}
\label{plan}
Sur le plan projectif, la condition ``position $(r+1)$-g\'en\'erale'' s'exprime en termes de droites $(r+3)$-s\'ecantes. Plus pr\'ecis\'ement, si $Z$ est un  groupe de points de $\P^{2\vee}$, on a d'apr\`es l'\'egalit\'e (\ref{eq}) apparaissant  dans  la preuve de la proposition \ref{pro31} :
$$h^1(\mathcal{J}_Z(r+1)\otimes O_{x^{\vee}})\neq 0 \Leftrightarrow \mid x^{\vee}\cap Z \mid \ge r+3.$$
Par cons\'equent, si $Z$ ne poss\`ede pas de $(r+3)$-s\'ecante, le faisceau 
$p_*q^*\mathcal{J}_Z(r+1)$ est un fibr\'e de rang $(r+2)$ sur $\P^2$.  De plus, si $h^0(\mathcal{J}_Z(r+1))=0$, il est un fibr\'e de Steiner :
{\small
$$\begin{CD}
0 @>>> p_{*}q^{*}{\mathcal J}_{Z}(r+1)@>>> H^{1}({\mathcal J}_{Z}(r))\otimes O_{\P^{n}}(-1) @>>> H^{1}({\mathcal J}_{Z}(r+1))\otimes O_{\P^{n}} @>>> 0.
\end{CD}
$$}
La proposition suivante fait le lien entre droite instable $l$ et courbe passant par
$Z\cup l^{\vee}$.
\begin{pro}
\label{prop-cle}
Soient $Z$ un groupe de points de $\P^{2\vee}$ en position $(r+1)$-g\'en\'erale et  $l^{\vee}\notin Z$. Alors, $l\in W(E_{r+1}(Z))$ si et seulement si 
$h^0(\mathcal{J}_{Z\cup l^{\vee}}(r+2))\neq 0.$
\end{pro}
\begin{rem} \upshape
Si le cardinal de $Z$ est strictement plus petit que $\binom{r+4}{2}-1$ il r\'esulte de cette proposition que toute droite du plan est instable. Par ailleurs, il appara\^{\i}t que si $Z$ est contenu dans une courbe de degr\'e $r+2$, alors tout point de cette courbe appartient \`a $W(E_{r+1}(Z))$.
\end{rem}
\begin{proof}
Notons $\widehat{\P}$ l'\'eclatement  de $\P^{2\vee}$ le long du point $l^{\vee}$. Rappelons  que 
$\widehat{\P}\simeq p^{-1}(l)\subset \F$ et consid\'erons le diagramme d'incidence induit :
$$
\begin{CD}
 \widehat{\P} @>\widehat{q}>>    \P^{2 \vee}\\
 @V\widehat{p}VV \\
l
\end{CD}
$$
Comme $Z$ n'a pas de $r+3$ s\'ecante on a 
(on pourra se r\'ef\'erer \`a la prop. 2.1 de \cite{V} pour plus de d\'etails et bien-s\^ur au livre \cite{OSS}) :
$$ p_*q^*\mathcal{J}_Z(r+1)\otimes O_l=\widehat{p}_*\widehat{q}^*\mathcal{J}_Z(r+1).$$
Ainsi la droite $l$ est instable pour le fibr\'e $p_*q^*\mathcal{J}_Z(r+1)$ si et seulement s'il existe un homomorphisme injectif 
$$ O_l(-1)\hookrightarrow \widehat{p}_*\widehat{q}^*\mathcal{J}_Z(r+1).$$
Ceci \'equivaut \`a la donn\'ee d'une application sur $\widehat{\P}$
$$ \widehat{p}^*O_l(-1)\hookrightarrow \widehat{q}^*\mathcal{J}_Z(r+1)$$
ou encore, si l'on note $\mathfrak{m}_{l^{\vee}}$ l'id\'eal maximal du point $l^{\vee}$, \`a la  donn\'ee d'une section non nulle 
$$ O_{\P^{2\vee}}\hookrightarrow \mathcal{J}_Z(r+1)\otimes \mathfrak{m}_{l^{\vee}}(1)=\mathcal{J}_{Z\cup l^{\vee}}(r+2).$$
\end{proof}
On peut reformuler cette proposition en termes de fibr\'es vectoriels de rang deux.
\begin{pro}
\label{fibextension}
Soient $Z$ un groupe de points de $\P^{2\vee}$ en position $(r+1)$-g\'en\'erale et $F$ une extension g\'en\'erale 
$$\begin{CD}
0 @>>> O_{\P^{2\vee}} @>>> F @>>> {\mathcal J}_{Z}(r+2) @>>> 0.
\end{CD}
$$
Alors $E_{r+1}(Z)^{\vee}(-1)=p_{*}q^{*}(F(-1)). $
De plus  $l\in W(E_{r+1}(Z))$ si et seulement s'il existe une section non nulle  de $F$ qui s'annule au point $l^{\vee}$.
\end{pro}
\begin{proof}
Comme  $p_*q^{*}O_{\P^{2\vee}}(-1)=0$ et $R^1p_*q^{*}O_{\P^{2\vee}}(-1)=0$,   on a bien 
$E_{r+1}(Z)^{\vee}(-1)=p_{*}q^{*}(F(-1))$. L'existence d'une droite instable $l$ fournit une application injective 
$$ O_l(-1)\hookrightarrow \widehat{p}_*\widehat{q}^*F(-1).$$
Ceci \'equivaut \`a la donn\'ee d'une appplication sur $\widehat{\P}$
$$ \widehat{p}^*O_l(-1)\hookrightarrow \widehat{q}^*F(-1)$$
ou encore, si l'on note $\mathfrak{m}_{l^{\vee}}$ l'id\'eal maximal du point $l^{\vee}$, \`a la  donn\'ee d'une section non nulle 
$$ O_{\P^{2\vee}}\hookrightarrow F(-1)\otimes \mathfrak{m}_{l^{\vee}}(1).$$
On peut remarquer que  si $l^{\vee}\notin Z$ on aura $h^0(F)\ge 2$, c'est-\`a-dire $h^0({\mathcal J}_{Z}(r+2))\neq 0$.
\end{proof}

\begin{thm}
\label{teo}
Soient $Z$ et $Z^{'}$ deux groupes de points de $\P^{2\vee}$ de m\^eme longueur et en position $(r+1)$-g\'en\'erale. On suppose  $E_{r+1}(Z)\simeq E_{r+1}(Z^{'})$. Alors un des deux cas suivants se produit :\\
1) $Z=Z^{'}$.\\
2) $Z$ et $Z^{'}$ sont sur une m\^eme courbe  $X_{r+2}$ de degr\'e $r+2$ et il existe un faisceau $\mathcal{L}$ de rang $1$ sur $X_{r+2}$ tel que $E_{r+1}(Z)\simeq E(X_{r+2}, \mathcal{L}).$
\end{thm}
\begin{proof}
Si $Z\neq Z^{'}$ il existe au moins une droite $l$ telle que $l^{\vee}\notin Z$ et pourtant  $l^{\vee}\in W(E_{r+1}(Z))$. 
D'apr\`es la proposition \ref{prop-cle}, il existe une courbe de degr\'e $r+2$ contenant $Z$. Soient $X_{r+2}$ une telle courbe et $\lbrace f=0 \rbrace$ son \'equation. Elle induit  la suite exacte suivante :
$$\begin{CD}
0 @>>> O_{\P^{2\vee}}(-1) @>f>> \mathcal{J}_Z(r+1) @>>> \mathcal{L}
 @>>> 0
\end{CD}
$$
o\`u $ \mathcal{L}$ est  donn\'e par la restriction du faisceau d'id\'eaux $\mathcal{J}_Z(r+1)$ \`a la courbe $X_{r+2}$.
Comme $p_*q^{*}O_{\P^{2\vee}}(-1)=0$ et $R^1p_*q^{*}O_{\P^{2\vee}}(-1)=0$ on en d\'eduit que 
$$ p_*q^*\mathcal{J}_Z(r+1)=\overline{p}_*\overline{q}^*\mathcal{L}.$$
Ce qui prouve le th\'eor\`eme.

\end{proof}


\bigskip

Laboratoire de Math\'ematiques Appliqu\'ees
de Pau\\
Universit\'e de Pau et des Pays de l'Adour\\
 Avenue de l'Universit\'e\\
 64000 Pau (France)\\
 {email:  jean.valles@univ-pau.fr } 

\end{document}